\documentclass[american]{article}
\usepackage[T1]{fontenc}
\usepackage[latin9]{inputenc}
\usepackage{geometry}
\usepackage{babel}
\usepackage{amsmath}
\usepackage{amsthm}
\usepackage{amssymb}
\usepackage{setspace}
\usepackage[unicode=true]
 {hyperref}

\makeatletter
\theoremstyle{remark}
\newtheorem*{rem*}{\protect\remarkname}
\theoremstyle{plain}
\newtheorem{thm}{\protect\theoremname}[section]
\theoremstyle{definition}
\newtheorem{defn}[thm]{\protect\definitionname}
\theoremstyle{plain}
\newtheorem{lem}[thm]{\protect\lemmaname}
\theoremstyle{plain}
\newtheorem{prop}[thm]{\protect\propositionname}
\theoremstyle{remark}
\newtheorem{claim}[thm]{\protect\claimname}

\@ifundefined{date}{}{\date{}}
\makeatother

\providecommand{\claimname}{Claim}
\providecommand{\definitionname}{Definition}
\providecommand{\lemmaname}{Lemma}
\providecommand{\propositionname}{Proposition}
\providecommand{\remarkname}{Remark}
\providecommand{\theoremname}{Theorem}

\begin{document}
\title{\onehalfspacing{}\noindent Bias Implies Low Rank for Quartic Polynomials}
\author{\onehalfspacing{}\noindent Amichai Lampert}
\maketitle
\begin{abstract}
\begin{onehalfspace}
\noindent We investigate the structure of polynomials of degree four
in many variables over a fixed prime field $\mathbb{F}=\mathbb{F}_{p}$.
In \cite{GT} it was shown that if a polynomial $f:\mathbb{F}^{n}\rightarrow\mathbb{F}$
is poorly distributed, then it is a function of a few polynomials
of smaller degree. In \cite{HS} an effective bound was found for
$f$ of degree four: If $bias\left(f\right)\geq\delta$, then the
number of lower degree polynomials required is at most polynomial
in $1/\delta$ and $f$ has a simple presentation as a sum of their
products. We make a step towards showing that in fact the number of
lower degree polynomials required is at most log-polynomial in $1/\delta$,
with the same simple presentation of $f$. This result was a Master's
thesis supervised by T. Ziegler at the Hebrew University of Jerusalem,
submitted in October 2018. A log-polynomial bound for polynomials
of arbitrary degree was recently proved independently by Milicevic
and by Janzer.
\end{onehalfspace}
\end{abstract}
\begin{onehalfspace}
\noindent \tableofcontents{}
\end{onehalfspace}

\newpage{}
\begin{onehalfspace}

\section{Introduction}
\end{onehalfspace}

\begin{onehalfspace}
\noindent Throughout this paper $\mathbb{F=\mathbb{F}}_{p}$ is a
fixed prime field for some $p\geq5$.\\
\\
For a function $f:\mathbb{F}^{n}\rightarrow\mathbb{F}$ and a direction
$h\in\mathbb{F}^{n}$, the discrete derivative $\varDelta_{h}f:\mathbb{F}^{n}\rightarrow\mathbb{F}$
is defined by the formula $\varDelta_{h}f\left(x\right)=f\left(x+h\right)-f\left(x\right)$.
For $d<p$, we say that $f$ is a polynomial of degree at most $d$
if for all $h_{1},\ldots,h_{d+1}\in\mathbb{F}^{n}$ we have $\varDelta_{h_{1}}\ldots\varDelta_{h_{d+1}}f\equiv0$.
We say that $f$ is of degree $d$ and write $deg\left(f\right)=d$
if $d$ is the minimal integer with this property. For a vector space
$V$ we define
\[
\mathfrak{\mathcal{P}}_{d}\left(V\right):=\left\{ f:V\rightarrow\mathbb{F}|\,f\,is\,a\,polynomial\,of\,degree\,at\,most\,d\right\} .
\]

\end{onehalfspace}
\begin{rem*}
\begin{onehalfspace}
\noindent In the sequel, linear polynomials will be denoted by Greek
letters and quadratic polynomials by Roman letters. 
\end{onehalfspace}
\end{rem*}
\begin{defn}[Rank]
\begin{onehalfspace}
\noindent Let $f:V\longrightarrow\mathbb{F}$ be a polynomial of
degree $d$. If we have a presentation\\
$f\left(x\right)=\sum_{i=1}^{r}g_{i}\left(x\right)h_{i}\left(x\right)+g_{0}\left(x\right)$
with $deg\left(g_{i}\right),deg\left(h_{i}\right)<d$, then we say
$f$ has \textbf{rank} at most $r$. We say that $f$ has rank $r$
and write $rank\left(f\right)=r$ if $r$ is the minimal integer with
such a presentation.\footnote{This quantity is also called the Schmidt h-invariant.}
\end{onehalfspace}
\end{defn}

\begin{onehalfspace}
\noindent We measure the distribution of a function $f:\mathbb{F}^{n}\rightarrow\mathbb{F}$
using $bias\left(f\right):=\left|\mathbb{\mathbb{E}}_{\boldsymbol{x}\in\mathbb{F}^{n}}e_{p}\left(f\left(x\right)\right)\right|$,
where $e_{p}\left(j\right):=e^{2\pi ij/p}$. If $bias\left(f\right)$
is large, then its values are poorly distributed. \\
\\
An important result in the field of higher-order Fourier analysis
is that biased polynomials are low-rank (see \cite{GT}). Because
the proof is not quantitatively effective, it is a question of some
interest to try and produce effective quantitative bounds. For quadratic
polynomials, a well known classical result is the following (see e.g.
Lemma 1.6 in \cite{GT}):
\end{onehalfspace}
\begin{thm}
\begin{onehalfspace}
\noindent \label{thm:quad:bias}Let $f:\mathbb{F}^{n}\rightarrow\mathbb{F}$
be a polynomial of degree 2 with $bias\left(f\right)\geq\delta$.
then $rank\left(f\right)\leq2\log_{p}\left(1/\delta\right)$.
\end{onehalfspace}
\end{thm}

\begin{onehalfspace}
\noindent Haramaty and Shpilka showed (in \cite{HS}) that if $bias\left(f\right)\geq\delta$,
then $rank\left(f\right)$ is polynomial in $\log_{p}\left(1/\delta\right)$
or polynomial in $1/\delta$, when f is of degree 3 or 4, respectively.
In this paper we work towards improving the bound for polynomials
of degree 4. Our main theorem is the following:
\end{onehalfspace}
\begin{thm}
\begin{onehalfspace}
\noindent \label{thm:Main}Let $f:\mathbb{F}^{n}\rightarrow\mathbb{F}$
be a polynomial of degree 4 with $bias\left(f\right)\geq\delta$.
Then there exists a subspace $V\subset\mathbb{F}^{n}$ and quadratic
polynomials $Q_{1},\ldots Q_{N}\subset\mathcal{P}_{2}\left(V\right)$,
where both $codim\left(V\right),N$ are $poly\left(log_{p}\left(1/\delta\right)\right)$,
such that for all $x\in V$ we have
\[
Q_{1}\left(x\right)=\ldots=Q_{N}\left(x\right)=0\,\implies\,\varDelta_{x}\varDelta_{x}\varDelta_{x}\varDelta_{x}f=0.
\]
\end{onehalfspace}
\end{thm}

\begin{onehalfspace}
\noindent \textbf{Note: }We have the Taylor expansion $g\left(x\right)=\frac{1}{4!}\varDelta_{x}\varDelta_{x}\varDelta_{x}\varDelta_{x}f$
which satisfies $f-g\in\mathcal{P}_{3}\left(x\right)$, so $f,g$
both have the same rank. 
\end{onehalfspace}

Applying the Nullstellensatz of Kazhdan and Ziegler (Theorem 1.8 in
\cite{KZ}) together with this result, we can bound $rank\left(f\right)$. 
\begin{rem*}
\begin{onehalfspace}
\noindent We expect that with similar methods this result can be extended
to $f$ of higher degree. 
\end{onehalfspace}
\end{rem*}
\begin{rem*}
A similar bound was recently proved independently for polynomials
of arbitrary degree by L. Milicevic in \cite{M} and also by O. Janzer
in \cite{J}.
\end{rem*}
\begin{onehalfspace}
\noindent \textbf{Note:} For a survey of higher-order Fourier analysis,
see \cite{HHL}.\\

\noindent The proof of Theorem \ref{thm:Main} will be composed of
several steps. Our starting point is a lemma from \cite{HS} which
says that we can restrict $f$ to a large subspace such that all of
its derivatives are low-rank. We then show that we can identify a
small number of quadratics which appear in all the derivatives. Finally,
we restrict our attention to the set of common zeros of these quadratics
and show that $\varDelta_{x}\varDelta_{x}\varDelta_{x}\varDelta_{x}f$
vanishes on this set. 
\end{onehalfspace}
\begin{onehalfspace}

\section{Identifying relevant quadratics}
\end{onehalfspace}

\begin{onehalfspace}
\noindent We begin with a lemma which allows us to restrict $f$ to
large subspaces.
\end{onehalfspace}
\begin{lem}
\begin{onehalfspace}
\noindent Let $V\subseteq\mathbb{F}^{n}$ be a subspace. Then $rank\left(f\right)\leq rank\left(f_{\mid V}\right)+codim\left(V\right)$.
\end{onehalfspace}
\end{lem}

\begin{proof}
\begin{onehalfspace}
\noindent Choose a basis so that $V=\left\{ x\in\mathbb{F}^{n}\mid x_{1}=x_{2}=\ldots=x_{k}=0\right\} $
where $k=codim\left(V\right)$. Write 
\[
f=\sum_{i=1}^{k}x_{i}g_{i}\left(x_{i+1},x_{i+2},\ldots,x_{n}\right)+g\left(x_{k+1},\ldots,x_{n}\right).
\]
Then $rank\left(f\right)\leq rank\left(g\right)+k=rank\left(f_{\mid V}\right)+codim\left(V\right)$.
\end{onehalfspace}
\end{proof}
\begin{onehalfspace}
\noindent In view of this lemma, it suffices to show that our polynomial
is low rank when restricted to a large subspace, a fact we will often
use. Setting $\rho:=\log_{p}\left(1/\delta\right)$, we begin with
Lemma 4.2 from \cite{HS}:
\end{onehalfspace}
\begin{lem}[Subspace with low rank derivatives]
\begin{onehalfspace}
\noindent \label{lem:low-rank-subspace} Let $f:\mathbb{F}^{n}\rightarrow\mathbb{F}$
be a degree 4 polynomial such that $bias\left(f\right)\geq\delta$.
Then there exists a linear subspace $V\subseteq\mathbb{F}^{n}$ such
that $codim_{\mathbb{F}^{n}}\left(V\right)=poly\left(\rho\right)$
, and such that for every $y\in V$ we have $rank\left(\Delta_{y}f\right)=poly\left(\rho\right)$.
\end{onehalfspace}
\end{lem}

\begin{onehalfspace}
\noindent Now we restrict our attention to the subspace that we get
from Lemma \ref{lem:low-rank-subspace}. We know that $\forall x,t\in V$
we have:
\begin{equation}
f\left(x+t\right)-f\left(x\right)=\sum_{i=1}^{n}\alpha_{t}^{i}\left(x\right)P_{t}^{i}\left(x\right)+P_{t}^{0}\left(x\right),\label{eq:presentation}
\end{equation}

\noindent where $\alpha_{t}^{1},\ldots,\alpha_{t}^{n}$ are linear
functions, $P_{t}^{0},\ldots,P_{t}^{n}$ are quadratics, and $n=poly\left(\rho\right)$.
\\

\noindent We will show that this stems from the presence of a small
family of quadratics appearing in many of the derivatives. We want
to work with a high rank\textbf{ }family of quadratics, meaning:
\end{onehalfspace}
\begin{defn}
\begin{onehalfspace}
\noindent Let $\left(Q_{1},\ldots,Q_{N}\right)\subset\mathcal{P}_{2}\left(V\right)$
be a family of quadratics. We say that the family is \textbf{R-regular
}if for any scalars $a_{1},\ldots,a_{N}\in\mathbb{F}$ not all zero,
we have 
\[
rank\left(\sum_{i=1}^{N}a_{i}Q_{i}\right)\geq R.
\]
\end{onehalfspace}
\end{defn}

\begin{onehalfspace}
\noindent We will require the following lemma which allows us to generate
subspaces from positive density sets (Lemma 2.3 in \cite{HS}):
\end{onehalfspace}
\begin{lem}[Bogolyubov-Chang]
\begin{onehalfspace}
\noindent \label{lem:Bogolyubov-Chang-lemma.}Let $V$ be a vector
space and $E\subset V$ such that $\left|E\right|=\mu\cdot\left|V\right|$.
Then there exists $b=O\left(\log\left(1/\mu\right)\right)$ such that
$bE-bE$ contains a subspace $U$ with $codim_{V}\left(U\right)=O\left(\log\left(1/\mu\right)\right)$.
In addition, there exists $C=p^{-poly\left(\log\left(1/\mu\right)\right)}$
such that every element $t\in U$ has at least $C\left|U\right|^{2b-1}$
representations $t=y_{1}+\ldots+y_{b}-z_{1}-\ldots-z_{b}$ where $y_{1},\ldots,y_{b},z_{1},\ldots,z_{b}\in E$.
\end{onehalfspace}
\end{lem}

\begin{onehalfspace}
\noindent The main proposition we prove in this section is the following:
\end{onehalfspace}
\begin{prop}
\begin{onehalfspace}
\noindent \label{prop: common-quads}Let $f,\,V$ be as above. Then
there exists an R-regular collection of homogenous quadratics $Q_{1},\ldots,Q_{N}$
and a subspace $V_{1}\subseteq V$ such that $\forall x,t\in V_{1}$
we have 
\[
f\left(x+t\right)-f\left(x\right)=\sum_{i=1}^{N}\alpha_{t}^{i}\left(x\right)Q_{i}\left(x\right)+\sum_{i=1}^{m}\beta_{t}^{i}\left(x\right)\gamma_{t}^{i}\left(x\right)\delta_{t}^{i}\left(x\right)+Q_{t}^{0}\left(x\right),
\]

\noindent where $N=poly\left(\rho\right)$, $m=poly\left(\rho\right)$,
and $codim_{V}\left(V_{1}\right)=R\cdot poly\left(\rho\right)$.\\
\end{onehalfspace}
\end{prop}

\begin{onehalfspace}
\noindent To prove this proposition we will gradually find structure
in the derivatives of $f$ , replacing the arbitrary quadratic polynomials
appearing in Equation (\ref{eq:presentation}) by fixed quadratics
appearing in all the derivatives. We accomplish this by repeatedly
applying the following lemma:
\end{onehalfspace}
\begin{lem}
\begin{onehalfspace}
\noindent Suppose there's a subspace $U\subseteq V$ ,a set $F\subseteq V$,
and fixed quadratics $Q_{1},\ldots,Q_{M}\in\mathcal{P}_{2}\left(V\right)$
such that $\forall x\in U,\,t\in F$ we have:
\begin{align*}
f\left(x+t\right)-f\left(x\right) & =\sum_{i=1}^{m}\beta_{t}^{i}\left(x\right)R_{t}^{i}\left(x\right)+\sum_{i=1}^{M}\gamma_{t}^{i}\left(x\right)Q_{i}\left(x\right)+\sum_{i=1}^{l}\delta_{t}^{i}\left(x\right)q_{t}^{i}\left(x\right)+R_{t}^{0}\left(x\right),
\end{align*}

\noindent where $rank\left(q_{t}^{i}\right)\leq l$ .\\
Then there exists a subspace $W\subseteq U$, a set $E\subseteq F$,
and fixed quadratics $Q_{M+1},\ldots,Q_{M+n}\in\mathcal{P}_{2}\left(V\right)$
such that $\forall x\in W,t\in E$: 
\[
f\left(x+t\right)-f\left(x\right)=\sum_{i=1}^{m-1}\epsilon_{t}^{i}\left(x\right)S_{t}^{i}\left(x\right)+\sum_{i=1}^{M+n}\zeta_{t}^{i}\left(x\right)Q_{i}\left(x\right)+\sum_{i=1}^{l+1}\eta{}_{t}^{i}\left(x\right)p_{t}^{i}\left(x\right)+S_{t}^{0}\left(x\right),
\]

\noindent where $\left|E\right|\geq\frac{1}{p^{m}+1}\left|F\right|$
, $codim_{U}\left(W\right)\leq n+M$ , and $rank\left(p_{t}^{i}\right)\leq l+1$.
\end{onehalfspace}
\end{lem}

\begin{proof}
\begin{onehalfspace}
\noindent During the proof we will ignore lower order terms in our
equations (e.g. a quadratic equation will hold up to some linear function).
Using the identity $\Delta_{t}\Delta_{s}f\left(x\right)=\Delta_{s}\Delta_{t}f\left(x\right)$
for $x,s\in U,\,t\in F$ we get:
\begin{gather*}
\sum_{i=1}^{n}\left[\alpha_{s}^{i}\left(t\right)P_{s}^{i}\left(x\right)+\alpha_{s}^{i}\left(x\right)\Delta_{t}P_{s}^{i}\left(x\right)\right]=\sum_{i=1}^{m}\left[\beta_{t}^{i}\left(s\right)R_{t}^{i}\left(x\right)+\beta_{t}^{i}\left(x\right)\varDelta_{s}R_{t}^{i}\left(x\right)\right]\\
+\sum_{i=1}^{M}\left[\gamma_{t}^{i}\left(s\right)Q_{i}\left(x\right)+\gamma_{t}^{i}\left(x\right)\varDelta_{s}Q_{i}\left(x\right)\right]+\sum_{i=1}^{l}\left[\delta_{t}^{i}\left(s\right)q_{t}^{i}\left(x\right)+\delta_{t}^{i}\left(x\right)\varDelta_{s}q_{t}^{i}\left(x\right)\right].
\end{gather*}

\noindent Setting $Z=\left\{ \left(s,t\right)|\,s\in U,\,t\in F,\,\beta_{t}^{i}\left(s\right)=\begin{cases}
1 & i=m\\
0 & otherwise
\end{cases}\right\} ,$ one of the following must hold:\\
\textbf{Case 1:} $\mathbb{P}_{t\in F}\left[\beta_{t}^{m}\in span\left(\beta_{t}^{1},\ldots,\beta_{t}^{m-1}\right)\right]\geq\frac{1}{p^{m}+1}.$\\
\textbf{Case 2:} $\mathbb{P}_{s\in U,t\in F}\left[\left(s,t\right)\in Z\right]\geq\frac{1}{p^{m}+1}.$\\
This is because if the first inequality doesn't occur, then $A:=\left\{ t\in F|\,\beta_{t}^{m}\in span\left(\beta_{t}^{1},\ldots,\beta_{t}^{m-1}\right)\right\} $
satisfies $\mathbb{P}_{t\in F}\left[t\notin A\right]\geq1-\frac{1}{p^{m}+1}$
so we get: 
\begin{gather*}
\mathbb{P}_{s\in U,t\in F}\left[\left(s,t\right)\in Z\right]=\mathbb{P}_{s\in U,t\in F-A}\left[\left(s,t\right)\in Z\right]\cdot\mathbb{P}_{t\in F}\left[t\notin A\right]\\
\geq p^{-m}\left(1-\frac{1}{p^{m}+1}\right)=\frac{1}{p^{m}+1}.
\end{gather*}
We now analyze both possible cases:\textbf{}\\
\textbf{Case 1: }Suppose $\mathbb{P}_{t\in F}\left[\beta_{t}^{m}\in span\left(\beta_{t}^{1},\ldots,\beta_{t}^{m-1}\right)\right]\geq\frac{1}{p^{m}+1}$
. Setting $E=\left\{ t\in F|\,\beta_{t}^{m}\in span\left(\beta_{t}^{1},\ldots,\beta_{t}^{m-1}\right)\right\} $,
we get that $\left|E\right|\geq\frac{1}{p^{m}+1}\left|F\right|$.
For every $t\in E$ there exist $a_{t}^{1},\ldots,a_{t}^{m-1}\in\mathbb{F}$
such that $\beta_{t}^{m}=\sum_{i=1}^{m-1}a_{t}^{i}\beta_{t}^{i}$
. Plugging this in we get that for all $x\in U,\,t\in E$ we have
\begin{gather*}
f\left(x+t\right)-f\left(x\right)=\sum_{i=1}^{m-1}\beta_{t}^{i}\left(x\right)R_{t}^{i}\left(x\right)+R_{t}^{m}\left(x\right)\sum_{i=1}^{m-1}a_{t}^{i}\beta_{t}^{i}\left(x\right)+\sum_{i=1}^{M}\gamma_{t}^{i}\left(x\right)Q_{i}\left(x\right)+\sum_{i=1}^{l}\delta_{t}^{i}\left(x\right)q_{t}^{i}\left(x\right)+R_{t}^{0}\left(x\right)\\
=\sum_{i=1}^{m-1}\beta_{t}^{i}\left(x\right)\left[R_{t}^{i}\left(x\right)+a_{t}^{i}R_{t}^{m}\left(x\right)\right]+\sum_{i=1}^{M}\gamma_{t}^{i}\left(x\right)Q_{i}\left(x\right)+\sum_{i=1}^{l}\delta_{t}^{i}\left(x\right)q_{t}^{i}\left(x\right)+R_{t}^{0}\left(x\right),
\end{gather*}
which is what we wanted.\textbf{}\\
\textbf{Case 2:} Suppose $\mathbb{P}_{s\in U,t\in F}\left[\left(s,t\right)\in Z\right]\geq\frac{1}{p^{m}+1}$.
Then there exists some $s_{0}\in U$ such that $E:=\left\{ t\in F|\,\left(s_{0},t\right)\in Z\right\} $
satisfies $\left|E\right|\geq\frac{1}{p^{m}+1}\left|F\right|$. For
$x\in U,\,t\in E$ we have:
\begin{gather*}
\sum_{i=1}^{n}\left[\alpha_{s_{0}}^{i}\left(t\right)P_{s_{0}}^{i}\left(x\right)+\alpha_{s_{0}}^{i}\left(x\right)\Delta_{t}P_{s_{0}}^{i}\left(x\right)\right]=R_{t}^{m}\left(x\right)+\sum_{i=1}^{m}\beta_{t}^{i}\left(x\right)\varDelta_{s_{0}}R_{t}^{i}\left(x\right)\\
+\sum_{i=1}^{M}\left[\gamma_{t}^{i}\left(s_{0}\right)Q_{i}\left(x\right)+\gamma_{t}^{i}\left(x\right)\varDelta_{s_{0}}Q_{i}\left(x\right)\right]+\sum_{i=1}^{l}\left[\delta_{t}^{i}\left(s_{0}\right)q_{t}^{i}\left(x\right)+\delta_{t}^{i}\left(x\right)\varDelta_{s_{0}}q_{t}^{i}\left(x\right)\right].
\end{gather*}
 Setting $W=\left\{ x\in U|\,\alpha_{s_{0}}^{i}\left(x\right)=0,\,\varDelta_{s_{0}}Q_{i}\left(x\right)=0\,for\,all\,i\right\} $
we get that for $x\in W,\,t\in E$ we have
\[
R_{t}^{m}\left(x\right)=Q_{t}\left(x\right)-\sum_{i=1}^{m-1}\beta_{t}^{i}\left(x\right)\varDelta_{s_{0}}R_{t}^{i}\left(x\right)-\sum_{i=1}^{l}\delta_{t}^{i}\left(s_{0}\right)q_{t}^{i}\left(x\right)+q_{t}\left(x\right),
\]
 where $Q_{t}\in span\left(Q_{1},\ldots,Q_{M},P_{s_{0}}^{1},\ldots,P_{s_{0}}^{n}\right)$,
and $rank\left(q_{t}\right)\leq l+1$ .\\
Plugging this in we have that for $x\in W,\,t\in E$ 
\begin{gather*}
f\left(x+t\right)-f\left(x\right)=\sum_{i=1}^{m-1}\beta_{t}^{i}\left(x\right)\left[R_{t}^{i}\left(x\right)+\beta_{t}^{m}\left(x\right)\varDelta_{s_{0}}R_{t}^{i}\left(x\right)\right]+\sum_{i=1}^{M+n}\zeta_{t}^{i}\left(x\right)Q_{i}\left(x\right)\\
+\sum_{i=1}^{l}q_{t}^{i}\left(x\right)\left[\delta_{t}^{i}\left(x\right)+\delta_{t}^{i}\left(s_{0}\right)\beta_{t}^{m}\left(x\right)\right]+\beta_{t}^{m}\left(x\right)q_{t}\left(x\right),
\end{gather*}
 where $\zeta_{t}^{i}\left(x\right)$are linear functions. This is
in the desired form.
\end{onehalfspace}
\end{proof}
\begin{onehalfspace}
\noindent We are now ready to prove Proposition \ref{prop: common-quads}.
\end{onehalfspace}
\begin{proof}
\begin{onehalfspace}
\noindent Applying the above lemma $n$ times, we get a subspace $W\subseteq V$
, a set $E\subseteq V$ and quadratics $Q_{1},\ldots,Q_{N}$ such
that for $x\in W,\,t\in E$ we have 
\[
f\left(x+t\right)-f\left(x\right)=\sum_{i=1}^{N}\gamma_{t}^{i}\left(x\right)Q_{i}\left(x\right)+\sum_{i=1}^{l}\delta_{t}^{i}\left(x\right)q_{t}^{i}\left(x\right)+P_{t}^{0}\left(x\right),
\]
 where $codim_{V}\left(W\right),N,l,rank\left(q_{t}^{i}\right),\log\frac{\left|V\right|}{\left|E\right|}$
are all $poly\left(\rho\right)$. \\
Now we want to upgrade the set $E$ to a large subspace. By Lemma
\ref{lem:Bogolyubov-Chang-lemma.}, we can find some $b=poly\left(\rho\right)$
such that $bE-bE$ contains a subspace $U$ with $codim\left(U\right)=poly\left(\rho\right)$
. Using the identity $\varDelta_{t\pm s}f=\varDelta_{t}f\pm\varDelta_{s}f$
(up to lower degree terms) we get that for $x,t\in V_{1}:=U\cap W$
we have 
\[
f\left(x+t\right)-f\left(x\right)=\sum_{i=1}^{k}\gamma_{t}^{i}\left(x\right)Q_{i}\left(x\right)+\sum_{i=1}^{l}\delta_{t}^{i}\left(x\right)q_{t}^{i}\left(x\right)+P_{t}^{0}\left(x\right),
\]
 where the parameters $l,\,rank\left(q_{t}^{i}\right),\,codim_{V}\left(V_{1}\right)$
are all $poly\left(\rho\right)$. Expanding the linear functions appearing
in $q_{t}^{i}$ we get 
\[
f\left(x+t\right)-f\left(x\right)=\sum_{i=1}^{k}\gamma_{t}^{i}\left(x\right)Q_{i}\left(x\right)+\sum_{i=1}^{m}\delta_{t}^{i}\left(x\right)\epsilon_{t}^{i}\left(x\right)\zeta_{t}^{i}\left(x\right)+P_{t}^{0}\left(x\right),
\]
 where $m=poly\left(\rho\right)$. \\
To make $Q_{1},\ldots,Q_{N}$ R-regular, we can get rid of low rank
quadratics in the following fashion:\\
Suppose WLOG we have $a_{1},\ldots,a_{N-1}\in\mathbb{F}$ such that
\[
Q_{N}\left(x\right)=\sum_{i=1}^{N-1}a_{i}Q_{i}\left(x\right)+\sum_{i=1}^{R}\alpha_{i}\left(x\right)\beta_{i}\left(x\right).
\]
 By restricting to $V_{1}'=\left\{ x\in V_{1}|\,\alpha_{i}\left(x\right)=0\,for\,all\,i\right\} $
we reduce our dimension by R at most and for $x,t\in V_{1}'$ we get
\begin{gather*}
\sum_{i=1}^{N}\gamma_{t}^{i}\left(x\right)Q_{i}\left(x\right)=\sum_{i=1}^{N-1}\gamma_{t}^{i}\left(x\right)Q_{i}\left(x\right)+\gamma_{t}^{N}\left(x\right)\sum_{i=1}^{N-1}a_{i}Q_{i}\left(x\right)\\
=\sum_{i=1}^{N-1}\left[\gamma_{t}^{i}\left(x\right)+\gamma_{t}^{N}\left(x\right)a_{i}\right]Q_{i}\left(x\right),
\end{gather*}
which implies 
\[
f\left(x+t\right)-f\left(x\right)=\sum_{i=1}^{N-1}\tilde{\gamma}_{t}^{i}\left(x\right)Q_{i}\left(x\right)+\sum_{i=1}^{m}\delta_{t}^{i}\left(x\right)\epsilon_{t}^{i}\left(x\right)\zeta_{t}^{i}\left(x\right)+P_{t}^{0}\left(x\right).
\]
We can keep doing this until our collection is $R$-regular, overall
reducing the dimension of our subspace by\\
$RN=R\cdot poly\left(\rho\right)$ at most. This completes the proof
of Proposition \ref{prop: common-quads}. 
\end{onehalfspace}
\end{proof}
\begin{onehalfspace}

\section{Vanishing on the zero set of the quadratics}
\end{onehalfspace}

\begin{onehalfspace}
\noindent Now we restrict our attention to the set $X=\left\{ x\in V_{1}|\,Q_{1}\left(x\right)=\ldots=Q_{N}\left(x\right)=0\right\} $.We
will see that $X$ is a well behaved set in terms of counting various
configurations. We introduce here some notation which will be used
in this section:
\end{onehalfspace}
\begin{itemize}
\begin{onehalfspace}
\item For vectors $s,t$ we define $X_{t}:=X\cap\left(X-t\right)$ and $X_{s,t}:=X_{t}\cap X_{s}$.
\item We will use $\left(\cdot,\cdot\right)_{i}$ to denote the bilinear
form associated with the $i-th$ quadratic in our collection, i.e.
$\left(s,t\right)_{i}:=Q_{i}\left(s+t\right)-Q_{i}\left(s\right)-Q_{i}\left(t\right).$
\end{onehalfspace}
\end{itemize}
\begin{onehalfspace}
\noindent By Proposition \ref{prop: common-quads}, we know that for
$t\in V_{1},\,x\in X_{t}$ we have 
\[
f\left(x+t\right)-f\left(x\right)=\sum_{i=1}^{m}\delta_{t}^{i}\left(x\right)\epsilon_{t}^{i}\left(x\right)\zeta_{t}^{i}\left(x\right)+P_{t}^{0}\left(x\right).
\]

\noindent We will now show that by restricting to a large subspace,
our function in fact vanishes on $X$. This stage will comprise two
steps: the first is removing the cubic term in the derivative, and
the second is removing the quadratic term.
\end{onehalfspace}
\begin{onehalfspace}

\subsection{Removing the cubic term}
\end{onehalfspace}

\begin{onehalfspace}
\noindent We will need the following useful claim:
\end{onehalfspace}
\begin{claim}[Independence with respect to linear equations]
\begin{onehalfspace}
\noindent \label{claim:linear equations}Let $U$ be a vector space
and $Q_{1},\ldots,Q_{N}\in\mathcal{P}_{2}\left(U\right)$ a R-regular
collection of quadratics. Let $X=\left\{ x\in U|\,Q_{1}\left(x\right)=\ldots=Q_{N}\left(x\right)=0\right\} $.
Then for any affine subspace $A\subseteq U$ we have 
\[
\left|\mathbb{P}_{x\in U}\left(x\in A\cap X\right)-p^{-N}\mathbb{P}_{x\in U}\left(x\in A\right)\right|\leq p^{-R/2}.
\]
\end{onehalfspace}
\end{claim}

\begin{proof}
\begin{onehalfspace}
\noindent Write $A=\left\{ x\in U|\,l_{1}\left(x\right)=c_{1},\ldots,l_{n}\left(x\right)=c_{n}\right\} ,$
where $n=codim\left(A\right)$. Then using Fourier analysis we get
\begin{gather*}
\left|A\cap X\right|=\sum_{x\in U}\underset{{\scriptstyle \boldsymbol{a}\in\mathbb{F}^{N},\boldsymbol{b}\in\mathbb{F}^{n}}}{\mathbb{E}}e_{p}\left(\sum_{i=1}^{N}a_{i}Q_{i}\left(x\right)+\sum_{i=1}^{n}b_{i}\left(l_{i}\left(x\right)-c_{i}\right)\right)\\
=p^{-N}\left|A\right|+p^{-N}\sum_{0\neq\boldsymbol{a}\in\mathbb{F}^{N}}\underset{{\scriptstyle \boldsymbol{b}\in\mathbb{F}^{n}}}{\mathbb{E}}\sum_{x\in U}e_{p}\left(\sum_{i=1}^{N}a_{i}Q_{i}\left(x\right)+\sum_{i=1}^{n}b_{i}\left(l_{i}\left(x\right)-c_{i}\right)\right).
\end{gather*}
For any $0\neq\boldsymbol{a}\in\mathbb{F}^{N},\,\boldsymbol{b}\in\mathbb{F}^{n}$
we know that $rank\left(\sum_{i=1}^{N}a_{i}Q_{i}\left(x\right)+\sum_{i=1}^{n}b_{i}\left(l_{i}\left(x\right)-c_{i}\right)\right)\geq R$
so by Theorem \ref{thm:quad:bias} we have 
\begin{gather*}
\left|p^{-N}\sum_{0\neq\boldsymbol{a}\in\mathbb{F}^{N}}\underset{{\scriptstyle \boldsymbol{b}\in\mathbb{F}^{n}}}{\mathbb{E}}\sum_{x\in U}e_{p}\left(\sum_{i=1}^{N}a_{i}Q_{i}\left(x\right)+\sum_{i=1}^{n}b_{i}\left(l_{i}\left(x\right)-c_{i}\right)\right)\right|\leq\\
p^{-N}\sum_{0\neq\boldsymbol{a}\in\mathbb{F}^{N}}\underset{{\scriptstyle \boldsymbol{b}\in\mathbb{F}^{n}}}{\mathbb{E}}\left|\sum_{x\in U}e_{p}\left(\sum_{i=1}^{N}a_{i}Q_{i}\left(x\right)+\sum_{i=1}^{n}b_{i}\left(l_{i}\left(x\right)-c_{i}\right)\right)\right|\leq\\
p^{-N}\sum_{0\neq\boldsymbol{a}\in\mathbb{F}^{N}}\underset{{\scriptstyle \boldsymbol{b}\in\mathbb{F}^{n}}}{\mathbb{E}}p^{-R/2}\left|U\right|\leq p^{-R/2}\left|U\right|.
\end{gather*}
 After plugging this in to the previous equality we get 
\[
\left|\left|A\cap X\right|-p^{-N}\left|A\right|\right|\leq p^{-R/2}\left|U\right|,
\]
as claimed. 
\end{onehalfspace}
\end{proof}
\begin{prop}
\begin{onehalfspace}
\noindent \label{prop:remove-cubic} If $R=poly\left(\rho\right)$
is large enough, then there exists a subspace $V_{2}\subset V_{1}$
and a set $E\subset V_{1}$ such that for $t\in E,\,x\in V_{2}\cap X_{t}$
we have 
\[
f\left(x+t\right)-f\left(x\right)=P_{t}\left(x\right),
\]
 where $codim_{V_{1}}\left(V_{2}\right)$ and $\log\left(\left|V_{1}\right|/\left|E\right|\right)$
are $poly\left(\rho\right)$.
\end{onehalfspace}
\end{prop}

\begin{onehalfspace}
\noindent To prove this proposition we will gradually shorten the
cubic sum appearing in $f\left(x+t\right)-f\left(x\right)$. For a
symmetric matrix $A\in\mathbb{F}^{n\times n\times n}$, we define
$entries\left(A\right):=\left\{ \left(i,j,k\right)|\,a_{i,j,k}\neq0\right\} $. 
\end{onehalfspace}
\begin{lem}
\begin{onehalfspace}
\noindent \label{lem:reducing-cubic} Suppose there's a symmetric
matrix $A\in\mathbb{F}^{n\times n\times n}$ , a set $F\subseteq V_{1}$
, and a subspace $U\subseteq V_{1}$ satisfying $codim_{V_{1}}\left(U\right)\leq R-8N-4n$
with the following property:\\
For all $t\in F$ there are linear functions $\gamma_{t}^{1},\ldots,\gamma_{t}^{n}$
such that for all $x\in X_{t}\cap U$ we have 
\[
f\left(x+t\right)-f\left(x\right)=\sum_{i,j,k=1}^{n}a_{i,j,k}\gamma_{t}^{i}\left(x\right)\gamma_{t}^{j}\left(x\right)\gamma_{t}^{k}\left(x\right)+P_{t}^{0}\left(x\right).
\]
 \textbf{Then} there exists a symmetric matrix $B\in\mathbb{F}^{n'\times n'\times n'},$
a set $E\subseteq F$, and a subspace $W\subseteq U$ where for all
$t\in E$ we have linear functions $\beta_{t}^{1},\ldots,\beta_{t}^{n'}$
such that for all $x\in X_{t}\cap W$ we have
\[
f\left(x+t\right)-f\left(x\right)=\sum_{i,j,k=1}^{n'}b_{i,j,k}\beta_{t}^{i}\left(x\right)\beta_{t}^{j}\left(x\right)\beta_{t}^{k}\left(x\right)+Q_{t}^{0}\left(x\right).
\]
$B$ satisfies either $n'<n$ or $entries\left(B\right)<entries\left(A\right)$
(by lexicographical ordering),$\left|E\right|\geq\frac{1}{p^{4N+2n}+1}\left|F\right|,$
and $codim_{U}\left(W\right)\leq3m+N.$
\end{onehalfspace}
\end{lem}

\begin{proof}
\begin{onehalfspace}
\noindent We denote $r:=codim_{V_{1}}\left(U\right)$. \\
As usual, our equations will hold up to lower order terms. Applying
the identity $\varDelta_{s}\varDelta_{t}f=\varDelta_{t}\varDelta_{s}f$
for $t\in F,\,s\in U$ satisfying$\left(s,t\right)_{i}=0\,\forall i\in\left[N\right]$
we get that for $x\in X_{s,t}\cap U$ we have 
\begin{gather*}
\sum_{i=1}^{m}\left[\delta_{s}^{i}\left(t\right)\epsilon_{s}^{i}\left(x\right)\zeta_{s}^{i}\left(x\right)+\delta_{s}^{i}\left(x\right)\epsilon_{s}^{i}\left(t\right)\zeta_{s}^{i}\left(x\right)+\delta_{s}^{i}\left(x\right)\epsilon_{s}^{i}\left(x\right)\zeta_{s}^{i}\left(t\right)\right]\\
=\sum_{i,j,k=1}^{n}a_{i,j,k}\left[\gamma_{t}^{i}\left(s\right)\gamma_{t}^{j}\left(x\right)\gamma_{t}^{k}\left(x\right)+\gamma_{t}^{i}\left(x\right)\gamma_{t}^{j}\left(s\right)\gamma_{t}^{k}\left(x\right)+\gamma_{t}^{i}\left(x\right)\gamma_{t}^{j}\left(x\right)\gamma_{t}^{k}\left(s\right)\right].
\end{gather*}
Now let $\left(i_{0},j_{0},k_{0}\right)\in entries\left(A\right)$
be lexicographically maximal. \\
Setting $Z=\left\{ \left(t,s\right)|\,s\in U\cap X,\,t\in F,\,\left(t,s\right)_{i}=0\,\forall i\in\left[N\right],\,\gamma_{t}^{k}\left(s\right)=\begin{cases}
1 & k=k_{0}\\
0 & otherwise
\end{cases}\right\} ,$ one of the following must hold:\\
\textbf{Case 1:} $\mathbb{P}_{t\in F}\left[\gamma_{t}^{k_{0}}\in span\left(\gamma_{t}^{1},\ldots,\gamma_{t}^{k_{0}-1},\gamma_{t}^{k_{0}+1},\ldots,\gamma_{t}^{n},\left(t,\cdot\right)_{1},\ldots,\left(t,\cdot\right)_{N}\right)\right]\geq\frac{1}{p^{4N+2n}+1}.$
\\
\textbf{Case 2:} $\mathbb{P}_{s\in U,t\in F}\left[\left(t,s\right)\in Z\right]\geq\frac{1}{p^{4N+2n}+1}.$\\
This is because if the first inequality doesn't occur, then setting

\noindent 
\begin{align*}
L & :=\left\{ t\in F|\,\gamma_{t}^{k_{0}}\in span\left(\gamma_{t}^{1},\ldots,\gamma_{t}^{k_{0}-1},\gamma_{t}^{k_{0}+1},\ldots,\gamma_{t}^{n},\left(t,\cdot\right)_{1},\ldots,\left(t,\cdot\right)_{N}\right)\right\} ,
\end{align*}
we get
\begin{gather*}
\mathbb{P}_{s\in U,t\in F}\left[\left(t,s\right)\in Z\right]=\mathbb{P}_{s\in U,t\in F-L}\left[\left(t,s\right)\in Z\right]\cdot\mathbb{P}_{t\in F}\left[t\notin L\right]\\
\geq\left(p^{-N}p^{-\left(N+n\right)}-p^{-\left(R-r\right)/2}\right)\left(1-\frac{1}{p^{4N+2n}+1}\right)\geq\frac{1}{p^{4N+2n}+1},
\end{gather*}
where we used the fact that $Q_{1},\ldots,Q_{N}\in\mathcal{P}_{2}\left(U\right)$
is an $R-r\geq8N+4n$ regular collection and applied Claim \ref{claim:linear equations}
to any fixed $t\in F-B$ (For the second inequality we also use the
fact that $x-x^{2}\geq x^{2}$ for $0\leq x\leq0.5$).We analyze both
possible cases:\textbf{}\\
\textbf{Case 1: }Suppose $\mathbb{P}_{t\in F}\left[\gamma_{t}^{k_{0}}\in span\left(\gamma_{t}^{1},\ldots,\gamma_{t}^{k_{0}-1},\gamma_{t}^{k_{0}+1},\ldots,\gamma_{t}^{n},\left(t,\cdot\right)_{1},\ldots,\left(t,\cdot\right)_{N}\right)\right]\geq\frac{1}{p^{4N+2n}+1}$
. Then for any $t\in L$ there exist $c_{t}^{1},\ldots,c_{t}^{n},d_{t}^{1},\ldots,d_{t}^{N}$
such that 
\[
\gamma_{t}^{k_{0}}=\sum_{k\neq k_{0}}c_{t}^{i}\gamma_{t}^{i}+\sum_{i=1}^{N}d_{t}^{i}\left(t,\cdot\right)_{i}.
\]
 Note that for $x\in X_{t}$ we have $\left(t,x\right)_{i}=Q_{i}\left(x+t\right)-Q_{i}\left(x\right)-Q_{i}\left(t\right)=-Q_{i}\left(t\right)$
so the functions $\left(t,\cdot\right)_{i}$ are constant on $X_{t}$
and therefore are swallowed by the lower order term. So for all $t\in L,\,x\in X_{t}\cap U$
we can express $f\left(x+t\right)-f\left(x\right)$ without $\gamma_{t}^{k_{0}}$,
i.e. for $E=L$ we get the desired result with $n'<n.$\\
\textbf{Case 2: }Suppose $\mathbb{P}_{s\in U,t\in F}\left[\left(t,s\right)\in Z\right]\geq\frac{1}{p^{4N+2n}+1}$.
Then there exists some $s_{0}\in U$ such that $\mathbb{P}_{t\in F}\left[\left(t,s_{0}\right)\in Z\right]\geq\frac{1}{p^{4N+2n}+1}$.
Let $E=\left\{ t\in F|\,\left(t,s_{0}\right)\in Z\right\} $. For
all $t\in E,\,x\in X_{s_{0},t}\cap U$ we have 
\begin{gather*}
\sum_{i=1}^{m}\left[\delta_{s_{0}}^{i}\left(t\right)\epsilon_{s_{0}}^{i}\left(x\right)\zeta_{s_{0}}^{i}\left(x\right)+\delta_{s_{0}}^{i}\left(x\right)\epsilon_{s_{0}}^{i}\left(t\right)\zeta_{s_{0}}^{i}\left(x\right)+\delta_{s_{0}}^{i}\left(x\right)\epsilon_{s_{0}}^{i}\left(x\right)\zeta_{s_{0}}^{i}\left(t\right)\right]\\
=\sum_{i,j,k=1}^{n}a_{i,j,k}\left[1_{i=k_{0}}\gamma_{t}^{j}\left(x\right)\gamma_{t}^{k}\left(x\right)+\gamma_{t}^{i}\left(x\right)1_{j=k_{0}}\gamma_{t}^{k}\left(x\right)+\gamma_{t}^{i}\left(x\right)\gamma_{t}^{j}\left(x\right)1_{k=k_{0}}\right].
\end{gather*}
Setting $W=\left\{ x\in U|\,\delta_{s_{0}}^{i}\left(x\right)=\epsilon_{s_{0}}^{i}\left(x\right)=\zeta_{s_{0}}^{i}\left(x\right)=0,\,\left(s_{0},x\right)_{i}=0\,for\,all\,i\right\} $
, we have $codim_{U}\left(W\right)\leq3m+N$ and for all $t\in E,\,x\in X_{t}\cap W$
(Using the fact that this implies $x\in X_{s_{0}}$) we get:
\[
\sum_{i,j,k=1}^{n}a_{i,j,k}\left[1_{i=k_{0}}\gamma_{t}^{j}\left(x\right)\gamma_{t}^{k}\left(x\right)+\gamma_{t}^{i}\left(x\right)1_{j=k_{0}}\gamma_{t}^{k}\left(x\right)+\gamma_{t}^{i}\left(x\right)\gamma_{t}^{j}\left(x\right)1_{k=k_{0}}\right]=0.
\]
Now, since $a_{i_{0},j_{0},k_{0}}\neq0$, this means that we can express
$\gamma_{t}^{i_{0}}\left(x\right)\gamma_{t}^{j_{0}}\left(x\right)$
as a linear combination of the other $\gamma_{t}^{i}\left(x\right)\gamma_{t}^{j}\left(x\right)$
for which $\left(i,j,k_{0}\right)\in entries\left(A\right)$. Rewriting
$f\left(x+t\right)-f\left(x\right)$ in this fashion the resulting
symmetric matrix $B\in\mathbb{F}^{n\times n\times n}$ satisfies $entries\left(B\right)<entries\left(A\right)$. 
\end{onehalfspace}
\end{proof}
\begin{onehalfspace}
\noindent We are now ready to prove Proposition \ref{prop:remove-cubic}.
\end{onehalfspace}
\begin{proof}
\begin{onehalfspace}
\noindent Asumming our quadratics are R-regular with $R=poly\left(\rho\right)$
large enough, we can apply Lemma \ref{lem:reducing-cubic} $\left(3m\right)^{3}$
times since every time the condition $codim_{V_{1}}\left(U\right)\leq R-8N-4n$
will be met every time. After these repeated applications we'll be
left with a set $E\subset V_{1}$ and a subspace $V_{2}\subset V_{1}$
such that $\forall t\in E,\,x\in X_{t}\cap V_{2}$ we have
\[
f\left(x+t\right)-f\left(x\right)=P_{t}\left(x\right),
\]

\noindent where $P_{t}\left(x\right)$ is some quadratic function.
By the bounds in the lemma, both $codim_{V_{1}}\left(V_{2}\right)$
and $\log\left(\left|V_{1}\right|/\left|E\right|\right)$ are $poly\left(\rho\right).$
\end{onehalfspace}
\end{proof}
\begin{onehalfspace}

\subsection{Removing the quadratic term and completing the proof}
\end{onehalfspace}

\begin{onehalfspace}
\noindent Proposition \ref{prop:remove-cubic} shows that, restricted
to $X$, our function $f$ behaves somewhat like a cubic function.
We will try to make this notion more concrete. 
\end{onehalfspace}
\begin{defn}
\begin{onehalfspace}
\noindent Let $\boldsymbol{h}\in V_{1}^{d}$. We say that $\boldsymbol{h}$
is \textbf{admissible} if for $n\in\left[N\right],\,i,j\in\left[d\right],\,i\neq j$
we have $\left(h_{i},h_{j}\right)_{n}=0$. 
\end{onehalfspace}
\end{defn}

\begin{rem*}
\begin{onehalfspace}
\noindent If there exists $x\in V_{1}$ such that $\left\{ x+\omega\cdot\boldsymbol{h}|\,\omega\in\left\{ 0,1\right\} ^{d}\right\} \subset X$
then $\boldsymbol{h}$ is necessarily admissible.
\end{onehalfspace}
\end{rem*}
\begin{defn}
\begin{onehalfspace}
\noindent Let $F\subset V_{1}$ be a subset, $W\subset V_{1}$ a subspace.
We say that $f$ is $\left(F,\,W\right)-cubic$ if $\forall t\in F,\,\boldsymbol{h}\in W^{3}$
such that $\left(t,\boldsymbol{h}\right)$ is admissible, we have
\[
\varDelta_{h_{1}}\varDelta_{h_{2}}\varDelta_{h_{3}}\varDelta_{t}f=0.
\]

\noindent If this holds for $\varepsilon$-a.e. admissible $\left(t,\boldsymbol{h}\right)\in F\times W^{3}$
we say that $f$ is $\varepsilon$-a.e. $\left(F,\,W\right)-cubic$. 
\end{onehalfspace}
\end{defn}

\begin{onehalfspace}
\noindent In order to apply the tools of Fourier analysis, the set
of admissible parallelepipeds must be large.
\end{onehalfspace}
\begin{claim}[Density of admissible parallelepipeds]
\begin{onehalfspace}
\noindent \label{claim:admissible-density} Let $F\subset V_{1}$
be a subset and $W\subset V_{1}$ a subspace with $\mu=\left|F\right|/\left|V_{1}\right|$
and $r=codim_{V_{1}}\left(W\right)$. Then

\noindent 
\[
\mathbb{P}_{t\in V_{1},\boldsymbol{h}\in V_{1}^{3}}\left(\left(t,\boldsymbol{h}\right)\in F\times W^{3}\,and\,is\,admissible\right)\geq p^{-6N-3r}\mu.
\]
\end{onehalfspace}
\end{claim}

\begin{proof}
\begin{onehalfspace}
\noindent We calculate 
\begin{gather*}
\mathbb{P}_{t\in V_{1},\boldsymbol{h}\in V_{1}^{3}}\left(\left(t,\boldsymbol{h}\right)\in F\times W^{3}\,and\,is\,admissible\right)=\\
\mathbb{P}_{t\in V_{1},\boldsymbol{h}\in V_{1}^{3}}\left(\left(t,\boldsymbol{h}\right)\in F\times W^{3}\,and\,is\,admissible|\,t\in F\right)\cdot\mathbb{P}_{t\in V_{1}}\left(t\in F\right)\\
\geq p^{-N-r}p^{-2N-r}p^{-3N-r}\mu=p^{-6N-3r}\mu.
\end{gather*}

\noindent The inequality follows from choosing $h_{1},h_{2},h_{3}\in W$
one after the other such that $\left(t,h_{1}\right),\left(t,h_{1},h_{2}\right),\left(t,h_{1},h_{2}.h_{3}\right)$
are all admissible.
\end{onehalfspace}
\end{proof}
\begin{onehalfspace}
\noindent We can now make the notion of cubic behavior more tangible.
Set $\mu=\left|E\right|/\left|V_{1}\right|,r=codim_{V_{1}}\left(V_{2}\right).$
\end{onehalfspace}
\begin{lem}
\begin{onehalfspace}
\noindent Let $\varepsilon>0$. If $R=poly\left(\rho,\log_{p}\left(1/\varepsilon\right)\right)$
is large enough, then $f$ is $\varepsilon$-a.e. $\left(E,\,V_{2}\right)-cubic$.
\end{onehalfspace}
\end{lem}

\begin{proof}
\begin{onehalfspace}
\noindent We need to show that for a.e. admissible $\left(t,\boldsymbol{h}\right)\in E\times V_{2}^{3}$
we can find $x\in V_{2}$ such that\\
$\left\{ x+\omega\cdot\left(t,\boldsymbol{h}\right)|\,\omega\in\left\{ 0,1\right\} ^{4}\right\} \subset X$.
This is enough because if we find suitable $x$ then by Proposition
\ref{prop:remove-cubic} we get 
\[
\varDelta_{h_{1}}\varDelta_{h_{2}}\varDelta_{h_{3}}\varDelta_{t}f=\varDelta_{h_{1}}\varDelta_{h_{2}}\varDelta_{h_{3}}P_{t}\left(x\right)=0.
\]
 Since $\left(t,\boldsymbol{h}\right)$ is admissible, it's sufficient
to find $x\in V_{2}$ such that $x,x+t,x+h_{1},x+h_{2},x+h_{3}\in X$
and automatically we get $\left\{ x+\omega\cdot\left(t,\boldsymbol{h}\right)|\,\omega\in\left\{ 0,1\right\} ^{4}\right\} \subset X.$

\noindent By Fourier analysis, the density of such $x$ is
\begin{gather*}
\underset{{\scriptstyle x\in V_{2}}}{\mathbb{E}}1_{X}\left(x\right)1_{X}\left(x+t\right)1_{X}\left(x+h_{1}\right)1_{X}\left(x+h_{2}\right)1_{X}\left(x+h_{3}\right)=\\
\underset{{\scriptstyle x\in V_{2}}}{\mathbb{E}}\underset{{\scriptstyle \alpha,\beta,\gamma,\delta,\epsilon\in\mathbb{F}^{N}}}{\mathbb{E}}e_{p}\left[\sum_{i=1}^{N}\left(\alpha_{i}Q_{i}\left(x\right)+\beta_{i}Q_{i}\left(x+t\right)+\gamma_{i}Q_{i}\left(x+h_{1}\right)+\delta_{i}Q_{i}\left(x+h_{2}\right)+\epsilon_{i}Q_{i}\left(x+h_{3}\right)\right)\right]=\\
p^{-5N}+p^{-5N}\sum_{{\scriptstyle \alpha,\beta,\gamma,\delta,\epsilon\in\mathbb{F}^{N}}\,not\,all\,zero}\underset{{\scriptstyle x\in V_{2}}}{\mathbb{E}}e_{p}\left[\sum_{i=1}^{N}\left(\alpha_{i}Q_{i}\left(x\right)+\beta_{i}Q_{i}\left(x+t\right)+\gamma_{i}Q_{i}\left(x+h_{1}\right)+\delta_{i}Q_{i}\left(x+h_{2}\right)+\epsilon_{i}Q_{i}\left(x+h_{3}\right)\right)\right].
\end{gather*}

\noindent To show that this sum is positive a.e., it's enough to show
that for a.e. admissible $\left(t,\boldsymbol{h}\right)\in E\times V_{2}^{3}$,
we have 
\[
\left|\underset{{\scriptstyle x\in V_{2}}}{\mathbb{E}}e_{p}\left[\sum_{i=1}^{N}\left(\alpha_{i}Q_{i}\left(x\right)+\beta_{i}Q_{i}\left(x+t\right)+\gamma_{i}Q_{i}\left(x+h_{1}\right)+\delta_{i}Q_{i}\left(x+h_{2}\right)+\epsilon_{i}Q_{i}\left(x+h_{3}\right)\right)\right]\right|\leq\frac{1}{2}p^{-5N},
\]
 whenever$\alpha,\beta,\gamma,\delta,\epsilon\in\mathbb{F}^{N}$ are
not all zero. To see this, we calculate
\begin{gather*}
\left(\underset{{\scriptstyle t\in V_{1},\,\boldsymbol{h}\in V_{2}^{3}}}{\mathbb{E}}\left|\underset{{\scriptstyle x\in V_{2}}}{\mathbb{E}}e_{p}\left[\sum_{i=1}^{N}\left(\alpha_{i}Q_{i}\left(x\right)+\beta_{i}Q_{i}\left(x+t\right)+\gamma_{i}Q_{i}\left(x+h_{1}\right)+\delta_{i}Q_{i}\left(x+h_{2}\right)+\epsilon_{i}Q_{i}\left(x+h_{3}\right)\right)\right]\right|\right)^{2}\leq\\
\underset{{\scriptstyle t\in V_{1},\,\boldsymbol{h}\in V_{2}^{3}}}{\mathbb{E}}\left|\underset{{\scriptstyle x\in V_{2}}}{\mathbb{E}}e_{p}\left[\sum_{i=1}^{N}\left(\alpha_{i}Q_{i}\left(x\right)+\beta_{i}Q_{i}\left(x+t\right)+\gamma_{i}Q_{i}\left(x+h_{1}\right)+\delta_{i}Q_{i}\left(x+h_{2}\right)+\epsilon_{i}Q_{i}\left(x+h_{3}\right)\right)\right]\right|^{2}=\\
\underset{{\scriptstyle x,y\in V_{2}}}{\mathbb{E}}e_{p}\left[\sum_{i=1}^{N}\alpha_{i}\left(Q_{i}\left(x\right)-Q_{i}\left(y\right)\right)\right]\underset{{\scriptstyle t\in V_{1}}}{\mathbb{E}}e_{p}\left[\sum_{i=1}^{N}\beta_{i}\left(Q_{i}\left(x+t\right)-Q_{i}\left(y+t\right)\right)\right]\underset{{\scriptstyle h_{1}\in V_{2}}}{\mathbb{E}}e_{p}\left[\sum_{i=1}^{N}\gamma_{i}\left(Q_{i}\left(x+h_{1}\right)-Q_{i}\left(y+h_{1}\right)\right)\right]\cdot\\
\underset{{\scriptstyle h_{2}\in V_{2}}}{\mathbb{E}}e_{p}\left[\sum_{i=1}^{N}\delta_{i}\left(Q_{i}\left(x+h_{2}\right)-Q_{i}\left(y+h_{2}\right)\right)\right]\underset{{\scriptstyle h_{3}\in V_{2}}}{\mathbb{E}}e_{p}\left[\sum_{i=1}^{N}\epsilon_{i}\left(Q_{i}\left(x+h_{3}\right)-Q_{i}\left(y+h_{3}\right)\right)\right]
\end{gather*}

\noindent If $\beta,\gamma,\delta,\epsilon$ are all zero, then $\alpha\neq0$
and the above expression is 
\[
\underset{{\scriptstyle x,y\in V_{2}}}{\mathbb{E}}e_{p}\left[\sum_{i=1}^{N}\alpha_{i}\left(Q_{i}\left(x\right)-Q_{i}\left(y\right)\right)\right]=bias\left(\sum_{i=1}^{N}\alpha_{i}Q_{i}\right)^{2}\leq p^{-\left(R-r\right)}.
\]

\noindent If $\beta,\gamma,\delta,\epsilon$ are not all zero, then
the above expression is bounded above by
\begin{multline*}
\underset{{\scriptstyle u\in V_{2}}}{\mathbb{E}}\left|\underset{{\scriptstyle t\in V_{1}}}{\mathbb{E}}e_{p}\left[\sum_{i=1}^{N}\beta_{i}\left(Q_{i}\left(t+u\right)-Q_{i}\left(t\right)\right)\right]\right|\cdot\left|\underset{{\scriptstyle h_{1}\in V_{2}}}{\mathbb{E}}e_{p}\left[\sum_{i=1}^{N}\gamma_{i}\left(Q_{i}\left(h_{1}+u\right)-Q_{i}\left(h_{1}\right)\right)\right]\right|\cdot\\
\left|\underset{{\scriptstyle h_{2}\in V_{2}}}{\mathbb{E}}e_{p}\left[\sum_{i=1}^{N}\delta_{i}\left(Q_{i}\left(h_{2}+u\right)-Q_{i}\left(h_{2}\right)\right)\right]\right|\cdot\left|\underset{{\scriptstyle h_{3}\in V_{2}}}{\mathbb{E}}e_{p}\left[\sum_{i=1}^{N}\epsilon_{i}\left(Q_{i}\left(h_{3}+u\right)-Q_{i}\left(h_{3}\right)\right)\right]\right|\leq p^{-\left(R-r\right)/2}.
\end{multline*}

\noindent So whenever $\alpha,\beta,\gamma,\delta,\epsilon\in\mathbb{F}^{N}$
are not all zero, we have
\[
\underset{{\scriptstyle t\in V_{1},\boldsymbol{h}\in V_{2}^{3}}}{\mathbb{E}}\left|\underset{{\scriptstyle x\in V_{2}}}{\mathbb{E}}e_{p}\left[\sum_{i=1}^{N}\left(\alpha_{i}Q_{i}\left(x\right)+\beta_{i}Q_{i}\left(x+t\right)+\gamma_{i}Q_{i}\left(x+h_{1}\right)+\delta_{i}Q_{i}\left(x+h_{2}\right)+\epsilon_{i}Q_{i}\left(x+h_{3}\right)\right)\right]\right|\leq p^{-\left(R-r\right)/2}.
\]

\noindent It follows that 
\[
\underset{{\scriptstyle t\in V_{1},\boldsymbol{h}\in V_{2}^{3}}}{\mathbb{P}}\left(\left|\underset{{\scriptstyle x\in V_{2}}}{\mathbb{E}}e_{p}\left[\sum_{i=1}^{N}\left(\alpha_{i}Q_{i}\left(x\right)+\beta_{i}Q_{i}\left(x+t\right)+\gamma_{i}Q_{i}\left(x+h_{1}\right)+\delta_{i}Q_{i}\left(x+h_{2}\right)+\epsilon_{i}Q_{i}\left(x+h_{3}\right)\right)\right]\right|\geq\frac{1}{2}p^{-5N}\right)\leq2p^{5N-\left(R-r\right)/2}.
\]
Setting $A:=\left\{ \left(t,\boldsymbol{h}\right)\in V_{1}\times V_{2}^{3}|\exists x\in V_{2}\,such\,that\,x,x+t,x+h_{1},x+h_{2},x+h_{3}\in X\right\} ,$and
taking the union over $\alpha,\beta,\gamma,\delta,\epsilon\in\mathbb{F}^{N}$
which are not all zero we get 
\[
\mathbb{P}_{t\in V_{1},\boldsymbol{h}\in V_{2}^{3}}\left(\left(t,\boldsymbol{h}\right)\notin A\right)\leq2p^{10N-\left(R-r\right)/2}.
\]
By Claim \ref{claim:admissible-density}, we find that 
\begin{align*}
\mathbb{P}_{t\in E,\boldsymbol{h}\in V_{2}^{3}}\left(\varDelta_{h_{1}}\varDelta_{h_{2}}\varDelta_{h_{3}}\varDelta_{t}f\neq0|\,\left(t,\boldsymbol{h}\right)\,is\,admissible\right) & \leq\frac{2}{\mu}p^{10N-\left(R-r\right)/2}p^{6N}=p^{C-R/2},
\end{align*}
 where $C=poly\left(\rho\right).$ This proves the claim.
\end{onehalfspace}
\end{proof}
\begin{onehalfspace}
\noindent In order to upgrade the set $E$ of ``good'' differences
to a subspace, we will use Lemma \ref{lem:Bogolyubov-Chang-lemma.}.
Applying the lemma with $E\subset V_{1}$, we denote the guaranteed
subspace by $U$ and set $V_{3}=U\cap V_{2}.$ Then $codim_{V_{1}}\left(V_{3}\right)$
is $poly\left(\rho\right).$ 
\end{onehalfspace}
\begin{lem}
\begin{onehalfspace}
\noindent Let $\varepsilon>0$. If $R=poly\left(\rho,\log_{p}\left(1/\varepsilon\right)\right)$
is large enough, then $f$ is $\varepsilon$-a.e. $\left(V_{3},\,V_{3}\right)-cubic$.
\end{onehalfspace}
\end{lem}

\begin{proof}
\begin{onehalfspace}
\noindent Since $\varDelta_{h_{1}}\varDelta_{h_{2}}\varDelta_{h_{3}}\varDelta_{s+s'}f=\varDelta_{h_{1}}\varDelta_{h_{2}}\varDelta_{h_{3}}\varDelta_{s}f+\varDelta_{h_{1}}\varDelta_{h_{2}}\varDelta_{h_{3}}\varDelta_{s'}f$,
it's enough to show that for a.e. admissible $\left(t,\boldsymbol{h}\right)\in V_{3}^{4}$
we can find many representations $t=r_{1}+\ldots+r_{c}-r_{c+1}-\ldots-r_{2c}$
such that $r_{1},\ldots,r_{2c}\in E$ and $\left(r_{i},\boldsymbol{h}\right)$
is admissible for every $i\in\left[2c-1\right]$ (in which case $\left(r_{2c},\boldsymbol{h}\right)$
must also be admissible). Call such a representation an $\boldsymbol{h}$-admissible
representation. By Fourier analysis, the density of such representations
is:
\begin{gather*}
\underset{{\scriptstyle \boldsymbol{\boldsymbol{r}}\in V_{1}^{2c-1}}}{\mathbb{E}}1_{E}\left(\sum_{i=1}^{c}r_{i}-\sum_{i=c+1}^{2c-1}r_{i}-t\right)\prod_{i=1}^{2c-1}1_{E}\left(r_{i}\right)\prod_{j=1}^{N}1_{\left(r_{i},h_{1}\right)_{j}=0}1_{\left(r_{i},h_{2}\right)_{j}=0}1_{\left(r_{i},h_{3}\right)_{j}=0}=\\
\underset{{\scriptstyle \boldsymbol{a,b,c}\in\mathbb{F}^{\left(2c-1\right)\times N}}}{\mathbb{E}}\underset{{\scriptstyle \boldsymbol{\boldsymbol{r}}\in V_{1}^{2c-1}}}{\mathbb{E}}1_{E}\left(\sum_{i=1}^{c}r_{i}-\sum_{i=c+1}^{2c-1}r_{i}-t\right)\prod_{i=1}^{2c-1}1_{E}\left(r_{i}\right)e_{p}\left[\sum_{i=1}^{2c-1}\sum_{j=1}^{N}\left(a_{i,j}\left(r_{i},h_{1}\right)_{j}+b_{i,j}\left(r_{i},h_{2}\right)_{j}+c_{i,j}\left(r_{i},h_{3}\right)_{j}\right)\right]=\\
p^{-3N\left(2c-1\right)}\underset{{\scriptstyle \boldsymbol{\boldsymbol{r}}\in V_{1}^{2c-1}}}{\mathbb{E}}1_{E}\left(\sum_{i=1}^{c}r_{i}-\sum_{i=c+1}^{2c-1}r_{i}-t\right)\prod_{i=1}^{2c-1}1_{E}\left(r_{i}\right)+p^{-3N\left(2c-1\right)}\cdot\\
\sum_{\boldsymbol{a,b,c}\in\mathbb{F}^{\left(2c-1\right)\times N}\,not\,all\,zero}\underset{{\scriptstyle \boldsymbol{\boldsymbol{r}}\in V_{1}^{2c-1}}}{\mathbb{E}}1_{E}\left(\sum_{i=1}^{c}r_{i}-\sum_{i=c+1}^{2c-1}r_{i}-t\right)\prod_{i=1}^{2c-1}1_{E}\left(r_{i}\right)e_{p}\left[\sum_{i=1}^{2c-1}\sum_{j=1}^{N}\left(a_{i,j}\left(r_{i},h_{1}\right)_{j}+b_{i,j}\left(r_{i},h_{2}\right)_{j}+c_{i,j}\left(r_{i},h_{3}\right)_{j}\right)\right].
\end{gather*}
By Lemma \ref{lem:Bogolyubov-Chang-lemma.}, for every $t\in V_{3}$
we have 
\[
p^{-3N\left(2c-1\right)}\underset{{\scriptstyle \boldsymbol{\boldsymbol{r}}\in V_{1}^{2c-1}}}{\mathbb{E}}1_{E}\left(\sum_{i=1}^{c}r_{i}-\sum_{i=c+1}^{2c-1}r_{i}-t\right)\prod_{i=1}^{2c-1}1_{E}\left(r_{i}\right)\geq p^{-D},
\]
where $D=poly\left(\rho\right).$ So in order to show there are many
such representations, it's enough to show that the contributions when
$\boldsymbol{a,b,c}\in\mathbb{F}^{\left(2c-1\right)\times N}$ are
not all zero are small for a.e. $\boldsymbol{h}\in V_{2}^{3}$. For
this, we calculate
\begin{gather*}
\left(\underset{{\scriptstyle \boldsymbol{h}\in V_{1}^{3}}}{\mathbb{E}}\left|\underset{{\scriptstyle \boldsymbol{\boldsymbol{r}}\in V_{1}^{2c-1}}}{\mathbb{E}}1_{E}\left(\sum_{i=1}^{c}r_{i}-\sum_{i=c+1}^{2c-1}r_{i}-t\right)\prod_{i=1}^{2c-1}1_{E}\left(r_{i}\right)e_{p}\left[\sum_{i=1}^{2c-1}\sum_{j=1}^{N}\left(a_{i,j}\left(r_{i},h_{1}\right)_{j}+b_{i,j}\left(r_{i},h_{2}\right)_{j}+c_{i,j}\left(r_{i},h_{3}\right)_{j}\right)\right]\right|\right)^{2}\leq\\
\underset{{\scriptstyle \boldsymbol{h}\in V_{1}^{3}}}{\mathbb{E}}\left|\underset{{\scriptstyle \boldsymbol{\boldsymbol{r}}\in V_{1}^{2c-1}}}{\mathbb{E}}1_{E}\left(\sum_{i=1}^{c}r_{i}-\sum_{i=c+1}^{2c-1}r_{i}-t\right)\prod_{i=1}^{2c-1}1_{E}\left(r_{i}\right)e_{p}\left[\sum_{i=1}^{2c-1}\sum_{j=1}^{N}\left(a_{i,j}\left(r_{i},h_{1}\right)_{j}+b_{i,j}\left(r_{i},h_{2}\right)_{j}+c_{i,j}\left(r_{i},h_{3}\right)_{j}\right)\right]\right|^{2}=\\
\underset{{\scriptstyle \boldsymbol{h}\in V_{1}^{3}}}{\mathbb{E}}\underset{{\scriptstyle \boldsymbol{r,s}\in V_{1}^{2c-1}}}{\mathbb{E}}1_{E}\left(\sum_{i=1}^{c}r_{i}-\sum_{i=c+1}^{2c-1}r_{i}-t\right)\prod_{i=1}^{2c-1}1_{E}\left(r_{i}\right)1_{E}\left(\sum_{i=1}^{c}s_{i}-\sum_{i=c+1}^{2c-1}s_{i}-t\right)\prod_{j=1}^{2c-1}1_{E}\left(s_{j}\right)\cdot\\
e_{p}\left[\sum_{i=1}^{2c-1}\sum_{j=1}^{N}\left(a_{i,j}\left(r_{i}-s_{i},h_{1}\right)_{j}+b_{i,j}\left(r_{i}-s_{i},h_{2}\right)_{j}+c_{i,j}\left(r_{i}-s_{i},h_{3}\right)_{j}\right)\right]\leq\\
\underset{{\scriptstyle \boldsymbol{r}\in V_{1}^{2c-1}}}{\mathbb{E}}\left|\underset{{\scriptstyle \boldsymbol{h}\in V_{1}^{3}}}{\mathbb{E}}e_{p}\left[\sum_{i=1}^{2c-1}\sum_{j=1}^{N}\left(a_{i,j}\left(r_{i},h_{1}\right)_{j}+b_{i,j}\left(r_{i},h_{2}\right)_{j}+c_{i,j}\left(r_{i},h_{3}\right)_{j}\right)\right]\right|.
\end{gather*}
We can bound this by squaring again:
\begin{gather*}
\left(\underset{{\scriptstyle \boldsymbol{r}\in V_{1}^{2c-1}}}{\mathbb{E}}\left|\underset{{\scriptstyle \boldsymbol{h}\in V_{1}^{3}}}{\mathbb{E}}e_{p}\left[\sum_{i=1}^{2c-1}\sum_{j=1}^{N}\left(a_{i,j}\left(r_{i},h_{1}\right)_{j}+b_{i,j}\left(r_{i},h_{2}\right)_{j}+c_{i,j}\left(r_{i},h_{3}\right)_{j}\right)\right]\right|\right)^{2}\leq\\
\underset{{\scriptstyle \boldsymbol{r}\in V_{1}^{2c-1}}}{\mathbb{E}}\left|\underset{{\scriptstyle \boldsymbol{h}\in V_{1}^{3}}}{\mathbb{E}}e_{p}\left[\sum_{i=1}^{2c-1}\sum_{j=1}^{N}\left(a_{i,j}\left(r_{i},h_{1}\right)_{j}+b_{i,j}\left(r_{i},h_{2}\right)_{j}+c_{i,j}\left(r_{i},h_{3}\right)_{j}\right)\right]\right|^{2}=\\
\underset{{\scriptstyle \boldsymbol{\boldsymbol{r}}\in V_{1}^{2c-1}}}{\mathbb{E}}\underset{{\scriptstyle \boldsymbol{h},\boldsymbol{h'}\in V_{1}^{3}}}{\mathbb{E}}e_{p}\left[\sum_{i=1}^{2c-1}\sum_{j=1}^{N}\left(a_{i,j}\left(r_{i},h_{1}-h'_{1}\right)_{j}+b_{i,j}\left(r_{i},h_{2}-h'_{2}\right)_{j}+c_{i,j}\left(r_{i},h_{3}-h'_{3}\right)_{j}\right)\right]\leq\\
\underset{{\scriptstyle \boldsymbol{h}\in V_{1}^{3}}}{\mathbb{E}}\left|\underset{{\scriptstyle \boldsymbol{\boldsymbol{r}}\in V_{1}^{2c-1}}}{\mathbb{E}}e_{p}\left[\sum_{i=1}^{2c-1}\sum_{j=1}^{N}\left(a_{i,j}\left(r_{i},h_{1}\right)_{j}+b_{i,j}\left(r_{i},h_{2}\right)_{j}+c_{i,j}\left(r_{i},h_{3}\right)_{j}\right)\right]\right|=\\
\underset{{\scriptstyle \boldsymbol{h}\in V_{1}^{3}}}{\mathbb{E}}\prod_{i=1}^{2c-1}\left|\underset{{\scriptstyle \boldsymbol{\boldsymbol{r}_{i}}\in V_{1}}}{\mathbb{E}}e_{p}\left[\sum_{j=1}^{N}\left(a_{i,j}\left(r_{i},h_{1}\right)_{j}+b_{i,j}\left(r_{i},h_{2}\right)_{j}+c_{i,j}\left(r_{i},h_{3}\right)_{j}\right)\right]\right|.
\end{gather*}
By Cauchy-Schwarzing twice we get
\begin{gather*}
\left|\underset{{\scriptstyle \boldsymbol{\boldsymbol{r}_{i}}\in V_{1}}}{\mathbb{E}}e_{p}\left[\sum_{j=1}^{N}\left(a_{i,j}\left(r_{i},h_{1}\right)_{j}+b_{i,j}\left(r_{i},h_{2}\right)_{j}+c_{i,j}\left(r_{i},h_{3}\right)_{j}\right)\right]\right|^{4}\leq\\
\left|\underset{{\scriptstyle \boldsymbol{\boldsymbol{r}_{i}}\in V_{1}}}{\mathbb{E}}e_{p}\left[\sum_{j=1}^{N}a_{i,j}\left(r_{i},h_{1}\right)_{j}\right]\right|\cdot\left|\underset{{\scriptstyle \boldsymbol{\boldsymbol{r}_{i}}\in V_{1}}}{\mathbb{E}}e_{p}\left[\sum_{j=1}^{N}b_{i,j}\left(r_{i},h_{2}\right)_{j}\right]\right|\cdot\left|\underset{{\scriptstyle \boldsymbol{\boldsymbol{r}_{i}}\in V_{1}}}{\mathbb{E}}e_{p}\left[\sum_{j=1}^{N}c_{i,j}\left(r_{i},h_{3}\right)_{j}\right]\right|.
\end{gather*}
Plugging this in yields
\begin{gather*}
\left[\underset{{\scriptstyle \boldsymbol{h}\in V_{1}^{3}}}{\mathbb{E}}\prod_{i=1}^{2c-1}\left|\underset{{\scriptstyle \boldsymbol{\boldsymbol{r}_{i}}\in V_{1}}}{\mathbb{E}}e_{p}\left[\sum_{j=1}^{N}\left(a_{i,j}\left(r_{i},h_{1}\right)_{j}+b_{i,j}\left(r_{i},h_{2}\right)_{j}+c_{i,j}\left(r_{i},h_{3}\right)_{j}\right)\right]\right|\right]^{4}\leq\\
\underset{{\scriptstyle \boldsymbol{h}\in V_{1}^{3}}}{\mathbb{E}}\prod_{i=1}^{2c-1}\left|\underset{{\scriptstyle \boldsymbol{\boldsymbol{r}_{i}}\in V_{1}}}{\mathbb{E}}e_{p}\left[\sum_{j=1}^{N}\left(a_{i,j}\left(r_{i},h_{1}\right)_{j}+b_{i,j}\left(r_{i},h_{2}\right)_{j}+c_{i,j}\left(r_{i},h_{3}\right)_{j}\right)\right]\right|^{4}\leq\\
\underset{{\scriptstyle \boldsymbol{h}\in V_{1}^{3}}}{\mathbb{E}}\prod_{i=1}^{2c-1}\left|\underset{{\scriptstyle \boldsymbol{\boldsymbol{r}_{i}}\in V_{1}}}{\mathbb{E}}e_{p}\left[\sum_{j=1}^{N}a_{i,j}\left(r_{i},h_{1}\right)_{j}\right]\right|\cdot\left|\underset{{\scriptstyle \boldsymbol{\boldsymbol{r}_{i}}\in V_{1}}}{\mathbb{E}}e_{p}\left[\sum_{j=1}^{N}b_{i,j}\left(r_{i},h_{2}\right)_{j}\right]\right|\cdot\left|\underset{{\scriptstyle \boldsymbol{\boldsymbol{r}_{i}}\in V_{1}}}{\mathbb{E}}e_{p}\left[\sum_{j=1}^{N}c_{i,j}\left(r_{i},h_{3}\right)_{j}\right]\right|\leq p^{-R},
\end{gather*}
so altogether we see that 
\[
\underset{{\scriptstyle \boldsymbol{h}\in V_{1}^{3}}}{\mathbb{E}}\left|\underset{{\scriptstyle \boldsymbol{\boldsymbol{r}}\in V_{1}^{2c-1}}}{\mathbb{E}}1_{E}\left(\sum_{i=1}^{c}r_{i}-\sum_{i=c+1}^{2c-1}r_{i}-t\right)\prod_{i=1}^{2c-1}1_{E}\left(r_{i}\right)e_{p}\left[\sum_{i=1}^{2c-1}\sum_{j=1}^{N}\left(a_{i,j}\left(r_{i},h_{1}\right)_{j}+b_{i,j}\left(r_{i},h_{2}\right)_{j}+c_{i,j}\left(r_{i},h_{3}\right)_{j}\right)\right]\right|\leq p^{-R/16}.
\]
Therefore we have
\begin{gather*}
\underset{{\scriptstyle \boldsymbol{h}\in V_{1}^{3}}}{\mathbb{P}}\left(\left|\underset{{\scriptstyle \boldsymbol{\boldsymbol{r}}\in V_{1}^{2c-1}}}{\mathbb{E}}1_{E}\left(\sum_{i=1}^{c}r_{i}-\sum_{i=c+1}^{2c-1}r_{i}-t\right)\prod_{i=1}^{2c-1}1_{E}\left(r_{i}\right)e_{p}\left[\sum_{i=1}^{2c-1}\sum_{j=1}^{N}\left(a_{i,j}\left(r_{i},h_{1}\right)_{j}+b_{i,j}\left(r_{i},h_{2}\right)_{j}+c_{i,j}\left(r_{i},h_{3}\right)_{j}\right)\right]\right|\geq\frac{1}{2}p^{-D}\right)\\
\leq2p^{D-R/16}.
\end{gather*}
Setting $A:=\left\{ \boldsymbol{h}\in V_{3}^{3}|\exists t\in V_{3}\,with\,density\leq\frac{1}{2}p^{-D}\,of\,h-admissible\,representations\right\} ,$
the union bound gives us $\mathbb{P}_{\boldsymbol{h}\in V_{3}^{3}}\left(\boldsymbol{h}\notin A\right)\leq2p^{E-R/16},$
with $E=poly\left(\rho\right).$ If $R=poly\left(\rho,\log\left(1/\varepsilon\right)\right)$
is large enough, this means that for $\varepsilon$-a.e. admissible
$\left(t,\boldsymbol{h}\right)\in V_{3}^{4}$, we have $\varDelta_{h_{1}}\varDelta_{h_{2}}\varDelta_{h_{3}}\varDelta_{t}f=0$.
\end{onehalfspace}
\end{proof}
\begin{lem}
\begin{onehalfspace}
\noindent If $\varepsilon=p^{-poly\left(\rho\right)}$ is small enough,
and $f$ is $\varepsilon$-a.e. $\left(V_{3},V_{3}\right)-cubic$,
then for \textbf{every }admissible $\boldsymbol{h}\in V_{3}^{4}$
we have $\varDelta_{h_{1}}\varDelta_{h_{2}}\varDelta_{h_{3}}\varDelta_{h_{4}}f=0$.
\end{onehalfspace}
\end{lem}

\begin{proof}
\begin{onehalfspace}
\noindent Let $\boldsymbol{h}\in V_{3}^{4}$ be admissible. If we
can find $\boldsymbol{t}\in V_{3}^{4}$ such that for all $\omega\in\left\{ 0,1\right\} ^{4}$
\begin{align*}
\varDelta_{t_{1}+\omega_{1}\left(h_{1}-2t_{1}\right)}\varDelta_{t_{2}+\omega_{2}\left(h_{2}-2t_{2}\right)}\varDelta_{t_{3}+\omega_{3}\left(h_{3}-2t_{3}\right)}\varDelta_{t_{4}+\omega_{4}\left(h_{4}-2t_{4}\right)}f & =0,
\end{align*}
then we get 
\[
\varDelta_{h_{1}}\varDelta_{h_{2}}\varDelta_{h_{3}}\varDelta_{h_{4}}f=\sum_{\omega\in\left\{ 0,1\right\} ^{4}}\varDelta_{t_{1}+\omega_{1}\left(h_{1}-2t_{1}\right)}\varDelta_{t_{2}+\omega_{2}\left(h_{2}-2t_{2}\right)}\varDelta_{t_{3}+\omega_{3}\left(h_{3}-2t_{3}\right)}\varDelta_{t_{4}+\omega_{4}\left(h_{4}-2t_{4}\right)}f=0.
\]
If we can show that there's a set of positive density of $\boldsymbol{t}\in V_{3}^{4}$
such that $\forall\omega\in\left\{ 0,1\right\} ^{4}$ the vectors
\begin{align*}
\left(t_{1}+\omega_{1}\left(h_{1}-2t_{1}\right),t_{2}+\omega_{2}\left(h_{2}-2t_{2}\right),t_{3}+\omega_{3}\left(h_{3}-2t_{3}\right),t_{4}+\omega_{4}\left(h_{4}-2t_{4}\right)\right)
\end{align*}

\noindent are admissible, then we're done.

\noindent Setting $W:=\left\{ t\in V_{3}|\,\left(t,h_{1}\right)_{i}=\left(t,h_{2}\right)_{i}=\left(t,h_{3}\right)_{i}=\left(t,h_{4}\right)_{i}=0\,for\,1\leq i\leq N\right\} $
, we see that any admissible $\boldsymbol{t}\in W^{4}$ will do the
job. 

\noindent By Claim (\ref{claim:admissible-density}) we have 
\begin{align*}
\mathbb{P}_{\boldsymbol{t}\in V_{3}^{4}}\left(\boldsymbol{t}\,is\,admissible,\,\boldsymbol{t}\in W^{4}\right) & \geq\left(\frac{\left|W\right|}{\left|V_{3}\right|}\right)^{4}{\displaystyle \mathbb{P}_{\boldsymbol{t}\in W^{4}}\left(\boldsymbol{t}\,is\,admissible\right)}\\
 & \geq p^{-16N}p^{-6N}=p^{-22N}.
\end{align*}

\noindent Therefore, if $\varepsilon=p^{-poly\left(\rho\right)}$
is small enough, we must have some admissible $\boldsymbol{t}\in W^{4}$
such that for all $\omega\in\left\{ 0,1\right\} ^{4}$
\[
\varDelta_{t_{1}+\omega_{1}\left(h_{1}-2t_{1}\right)}\varDelta_{t_{2}+\omega_{2}\left(h_{2}-2t_{2}\right)}\varDelta_{t_{3}+\omega_{3}\left(h_{3}-2t_{3}\right)}\varDelta_{t_{4}+\omega_{4}\left(h_{4}-2t_{4}\right)}f=0.
\]

\noindent which proves the claim.
\end{onehalfspace}
\end{proof}
\begin{onehalfspace}
\noindent We can now prove Theorem \ref{thm:Main}. By the results
of this section, if $Q_{1},\ldots Q_{N}$ are $R$-regular with $R=poly\left(\rho\right)$
then we are left with a subspace $V_{3}\subset\mathbb{F}^{n}$ such
that $codim\left(V_{3}\right)=poly\left(\rho\right)$ and for every
admissible $\boldsymbol{h}\in V_{3}^{4}$ we have $\varDelta_{h_{1}}\varDelta_{h_{2}}\varDelta_{h_{3}}\varDelta_{h_{4}}f=0.$
For any $x\in X\cap V_{3},$ $\left(x,x,x,x\right)\in V_{3}^{4}$
is admissible so we have $\varDelta_{x}\varDelta_{x}\varDelta_{x}\varDelta_{x}f=0$. 
\end{onehalfspace}

\end{document}